\documentclass[final]{llncs}

\usepackage[applemac]{inputenc}
\usepackage{times,amsmath,amssymb,tabularx,gastex,hhline,rotating,enumerate,xspace}
\usepackage[final]{hyperref}
\usepackage[usenames]{color}
\usepackage{showkeys}

\newcommand{\comment}[1]{}

\newenvironment{vd}{\noindent\color{magenta} VD :  }{}

\newenvironment{AU}{\noindent\color{red} AU : }{}

\newenvironment{jl}{\noindent\color{blue} JL :  }{}

\newcommand{\refthm}[1]{Thm.~\ref{#1}}
\newcommand{\refcor}[1]{Cor.~\ref{#1}}

\newcommand{\reflem}[1]{Lem.~\ref{#1}}
\newcommand{\refprop}[1]{Prop.~\ref{#1}}

\newcommand{\refsec}[1]{Sect.~\ref{#1}}

\newcommand{\refeq}[1]{Eq.~\ref{#1}}

\newenvironment{proofof}[2]{\textit{Proof of #1}: {#2} 
\hfill \qed}

\newcommand{\set}[2]{\left\{#1\mathrel{\left|\vphantom{#1}\vphantom{#2}\right.}#2\right\}}
\newcommand{\oneset}[1]{\left\{\mathinner{#1}\right\}}

\newcommand{\abs}[1]{\left|\mathinner{#1}\right|}
\newcommand{\Abs}[1]{\left\Vert\mathinner{#1}\right\Vert}

\newcommand{\floor}[1]{\left\lfloor\mathinner{#1} \right\rfloor}

\newcommand{\gen}[1]{\left< \mathinner{#1} \right>}

\newcommand{\N}{\mathbb{N}}
\newcommand{\Z}{\mathbb{Z}}

\renewcommand{\phi}{\varphi}
\newcommand{\eps}{\varepsilon}
\newcommand{\e}{\eps} 

\newcommand{\del}{\delta}

\newcommand{\sig}{\sigma}

\newcommand\DD{\Delta}
\newcommand\GG{\Gamma}
\newcommand\LL{\Lambda}

\newcommand{\Oh}{\mathcal{O}}

\newcommand{\cL}{\mathcal{L}}

\newcommand{\cT}{\mathcal{T}}

\renewcommand{\cT}{\mathcal{T}}
\newcommand{\BS}{\mathrm{\bf{BS}}}

\newcommand{\baum}{Baums\-lag\xspace}
\newcommand{\BG}{G_{(1,2)}}

\newcommand\s{\mathop\mathrm{swap}}

\newcommand\sdpz{\Z[1/2] \rtimes \Z}

\newcommand{\IFF}{if and only if\xspace}
\newcommand{\homo}{homomorphism\xspace}

\newcommand{\Breduced}{Britton-re\-du\-ced\xspace}

\newcommand{\Breduction}{Britton re\-duc\-tion\xspace}
\newcommand{\WP}{Word Problem\xspace}

\newcommand\lds{,\ldots ,} 
\newcommand{\sse}{\subseteq}
\newcommand{\es}{\emptyset}
\newcommand{\sm}{\setminus}
\newcommand{\OS}{\oneset{-1,0,+1}}

\newcommand\ei[1]{{\emph{#1}\xspace}\index{#1}} 

\newcommand{\proc}[1]{\ensuremath{\text{\textsc{#1}}}}
\newcommand{\clone}{\proc{clone}\xspace}
\newcommand{\ET}{\proc{ExtendTree}\xspace}
\newcommand{\MT}{\proc{MakeTree}\xspace}
\newcommand{\tow}{\tau} 
\newcommand{\cupi}{\dot\cup} 
\newcommand{\PC}{power circuit\xspace}

\newcommand{\MDS}{main data structure\xspace}
\newcommand{\tm}{triple marking\xspace}
\newcommand{\ch}{\ensuremath{\mathrm{c}}}
\newcommand{\pot}{\mathop{\mathrm{pot}}}

\begin{document}
\title{Efficient algorithms for highly compressed data: 
{T}he Word Problem in {H}igman's group is in {P}}
\author{Volker Diekert\inst{1}, J{\"u}rn Laun\inst{1}, Alexander Ushakov\inst{2}}
\institute{FMI, Universit\"at Stuttgart,
Universit\"atsstr. 38, D-70569 Stuttgart, Germany\and 
Department of Mathematics, Stevens Institute of Technology, Hoboken,  NJ 07030, USA}

\date{\today}

\maketitle
\begin{abstract}
Power circuits are data structures 
which support efficient algorithms for highly compressed integers. 
Using this new data structure  it has been shown recently by   Myasnikov, Ushakov and Won that the 
Word Problem of the one-relator \baum group is in P. Before that
the best known upper bound has been non-elementary. 
In the present  paper we provide new results for power circuits and 
we give new applications in algorithmic algebra and algorithmic group theory:
1.~We define a modified reduction procedure on power circuits which 
runs in quadratic time thereby improving the known cubic time complexity. The improvement is crucial for our other results. 
2.~We improve the complexity of the Word Problem for the \baum group
to cubic time thereby providing the first practical algorithm for that problem. 
3.~The main result is that the Word Problem of Higman's group is decidable in 
polynomial time. The situation for  Higman's group is more complicated than for the \baum group and forced us to advance the theory of \PC{}s. 
\end{abstract}
{\small {\bf Key words:}
Data structures, Compression, Algorithmic group theory, Word Problem

\section*{Introduction}\label{intro}

\emph{Power circuits} have been introduced in \cite{muw11pc}. It is a data structure for integers which supports
+, $-$, $\leq$, a restricted version of multiplication, and raising to the power of 2. 
Thus, by iteration it is possible to represent (huge) values involving the tower function by very small circuits.
Another way to say this is that efficient algorithms for \PC{}s yield efficient algorithms for arithmetic with integers in highly compressed form. This idea 
of \ei{efficient algorithms for highly compressed data} is the main underlying theme of the present paper. In this sense our paper is more about compression and
data structures than about algorithmic group theory. However, the applications are in this area so far. 

Indeed as a first application of \PC{}s, \cite{muw11bg} showed that the \WP of the \baum group\footnote{Sometimes called 
\baum-Gersten group, e.g.{} in \cite{plat04} or in a preliminary version of \cite{muw11bg}.}  is solvable in polynomial time. 
Algorithmic interests have a long history in 
combinatorial group theory. In 1910 Max Dehn \cite{dehn11} formulated fundamental 
algorithmic problems for (finitely presented)  groups. The most prominent one is the 
\ei{\WP:} "Given a finite presentation of some fixed group $G$, decide whether an input word 
$w$ represents the trivial element $1_G$ in $G$." 
It took more than four decades until Novikov and Boone showed (independently) in the 1950's the existence
of a fixed finitely presented group with an undecidable \WP, \cite{nov55,boone59}. 
It is also true that there are finitely presented groups with a decidable \WP but
with arbitrarily high complexity, \cite[Theorem 1.3]{brs02}. 
In these examples the difficult instances for the \WP are extremely sparse,  
(because they encode Turing machine computations) and, 
inherently due to the constructions, these groups never appear in any natural setting. 

In contrast, the \baum group $\BG$  is given by a single 
defining relation, see \refsec{wpg}. (It is a non-cyclic one-relator group all of whose finite
factor groups are cyclic \cite{baumslag69}.) 
 It has been a natural (and simplest) candidate for a group with a non-polynomial  \WP in the worst case, because the Dehn function\footnote{We do not use any result about Dehn functions here, and we refer the interested reader e.g.{} to Wikipedia (or to the appendix) for a formal definition.} of $\BG$ is 
non-elementary by a result due to Gersten
\cite{gersten91}, see also \cite{plat04}.
Moreover,
the only general way to solve the word problem in one-relator groups is by a Magnus break-down procedure \cite{mag32,LS01} which computes normal forms. 
It was developed in the 1930s and there is
no progress ever since.
Its time-complexity on $\BG$ is non-elementary, since it cannot be bounded by any tower of exponents.

So, the question of algorithmic hardness of the \WP in one-relator groups is still wide open.
Some researchers conjecture it is polynomial (even quadratic, see
\cite{BMS}), based on observations
on generic-case complexity \cite{KMSS1}. Others conjecture that it cannot be
polynomial, based on the fast growing Dehn functions. 
(Note that the Dehn function gives a lot information about the group.
E.g., if it is linear, then the group is hyperbolic, and
the \WP is linear. If it is computable, then the \WP is decidable \cite{MadlenerO88}.)

The contributions of the present paper are as follows: 
In a first part, we give new efficient manipulations of the data structure of \emph{\PC{s}}. 
Concretely, we define a new reduction procedure called \ET on power circuits. It improves the complexity 
of the reduction algorithm from cubic to quadratic time. This is our first result. It turns out to be essential,
because reduction is a fundamental tool and applied as a black-box operation
frequently. For example, with the help of a better reduction 
algorithm (and some other ideas) we can improve as our second result the complexity of the \WP in 
$\BG$ significantly from $\Oh(n^7)$ in \cite{muw11bg} down to $\Oh(n^3)$, 
\refthm{wpbg}. This cubic algorithm is the first practical algorithm 
which works for that problem on all reasonably short instances. 

The basic structure in our paper is the domain of rational numbers, where nominators are restricted to powers of two. 
Thus, we are working  in the ring $\Z[1/2]$. We view $\Z[1/2]$ as an Abelian 
group where multiplication with $1/2$ is an automorphism. This in turn can be embedded into a 
semi-direct product $\sdpz$ which is the set of pairs $(r,k) \in \Z[1/2] \times \Z$ 
with the multiplication $(r,k)(s,\ell) = (r + 2^k s, k+ \ell)$. There is a natural 
partially defined swap operation which interchanges the first and second component.
Semi-direct products (or more generally, wreath products) appear in various places as basic mathematical objects, and this makes  
the algebra  $\sdpz$ with swapping interesting in its own right. 
This algebra has a \WP in $\Oh(n^4)$ (\refthm{wpsdpz}). The \WP of $\BG$ can be understood as a special case with the better $\Oh(n^3)$ performance. 

Another new application of \PC{s} shows that the \WP in Higman's group $H_4$ is decidable 
in polynomial time. (We could also consider any $H_q$ with  $q \geq 4$.) This is our third and main result.  
Higman \cite{higman51} constructed $H_4$ in 1951 as the first example 
of a finitely presented 
 infinite group where all finite quotient groups are trivial.
This leads immediately to a finitely generated infinite simple quotient group of $H_4$;
and no such group was known before Higman's construction. The group 
$H_4$ is constructed by a few operations involving amalgamation (see e.g. \cite{serre80}). 
Hence, a Magnus break-down procedure (for amalgamated products)  yields decidability of the
\WP. The procedure computes normal forms, but the length of  normal forms 
can be a tower function in the input length. (More accurately, one can show that
the Dehn function of $H_4$ has an order of magnitude as a tower function \cite{bridson10}.) 
Thus, Higman's group has been another  natural, but rather complicated candidate for a finitely presented group with an extremely hard 
\WP. 
Our paper eliminates $H_4$ as a candidate: We show  that the \WP of $H_4$ is in $\Oh(n^6)$ (\refthm{wph}).

We  obtain this result by new techniques for efficient manipulations of  multiple markings 
in a single \PC and their ability for huge compression rates. Compression 
techniques have been applied elsewhere for solving word problems, \cite{Lohrey06siam,LohreyS07,schleimer08,HauboldL09}. But in these papers the authors 
use straight-line programs whose compression rates are far too small (at best exponential)
to cope with \baum or Higman groups. 

Due to lack of space some few proofs are shifted to the appendix. 
 
\section{Notation and preliminaries}\label{notation}
Algorithms and (decision) problems are classified by their \ei{time complexity}
on a random-access machine (RAM). Frequently we use the notion of \ei{amortized
analysis} with respect to a \ei{potential function}, see e.g. in \cite[Sect. 17.3]{CLRS09}. 

The \ei{tower function} $\tow:\N \to 2^\N$
is defined as usual: $\tow(0) = 1$ and $\tow(i+1) = 2^{\tow(i)}$ for $i \geq 0$. Thus, 
e.g. $\tow(4) = 2^{2^{2^{2^{1}}}}= 2^{16}$ and $\tow(6)$ written in binary requires more 
bits than there are supposed to be electrons in this universe. 

We use standard notation and facts from group theory as the reader can find in the 
classical text book \cite{LS01}. In particular, we apply the standard (so called
Magnus break-down) procedure for solving the word problem in HNN-extensions and
amalgamated products. All HNN-extensions 
and amalgamated products in this paper have an explicit finite presentation. 
Amalgamated products are denoted by $G*_A H$ where $A$ is a subgroup in $G$ and in 
$H$. The formal definition of $G*_A H$ creates first a disjoint copy $H'$ of 
$H$. Then one considers the free product $G* H$ and adds defining relations
identifying $a \in A$ with its copy $a' \in A'$. We refer to \cite{serre80} for 
the basic facts about Higman groups.

\section{Power circuits}\label{PCs}

This section is based on \cite{muw11pc}, but we also 
provide new material like our treatment of multiple markings
and improved time complexities.\footnote{In order to keep the paper 
self-contained and as 
we use a slightly different notation we give full proofs in the appendix.} 
Let $\GG$ be a set and $\del$ be a mapping $\del: \GG \times \GG\to \OS$. This defines
a directed graph $(\GG, \DD)$, where $\GG$ is the set of vertices and the set of directed arcs (or edges) is 
$\DD=\set{(P,Q)\in \GG \times \GG}{\del(P,Q)\neq 0}$ (the support of the mapping $\del$). 
Throughout we assume that $(\GG, \DD)$ is a \ei{dag} (\ei{directed acyclic graph}). 
In particular, $\del(P,P)=0$ for all vertices $P$. 

A \ei{marking} is a mapping $M:\GG\to\OS$. We can also think of a marking as a
subset of $\GG$ where each element in $M$ has a sign ($+$ or $-$). 
(Thus, we also speak about a \ei{signed subset}.)  
Each node $P\in \GG$ is associated in a natural way with a marking, which is called 
its $\LL$-marking $\LL_P$ and which is defined as follows: 
$$\LL_P: \GG\to \OS, \; Q \mapsto \del(P,Q)$$
Thus, the marking $\LL_P$ is the signed subset 
which corresponds to the 
targets of outgoing arcs from $P$. 

We define the \ei{evaluation} $\e(P)$ of a node ($\e(M)$ of a marking resp.)
bottom-up in the dag by induction:
\begin{align*}
\e(P) &= 2^{\e(\LL_P)} &\text{for a node $P$}, \\
\e(M) &= \sum_{P}M(P)\e(P) &\text{for a marking  $M$}.
\end{align*}
Note that leaves evaluate to $1$, the evaluation of a marking is a real number, and the 
evaluation of a node $P$ is a positive real number. Thus, $\e(P)$ and $\e(M)$
are well-defined. We have the following nice formula for nodes: 
$ \log_2(\e(P)) = \e(\LL_P)$.
Therefore we can view the marking $\LL_P$ as "taking the logarithm of $P$". 

\begin{definition}\label{def:PC}
A \ei{\PC} is a pair $\Pi=(\GG,\del)$ with $\del: \GG \times \GG\to \OS$ such that $(\GG, \DD)$ is a dag as above with the additional property 
that $\e(M)\in\Z$ for all markings $M$.
\end{definition}

We will see later in \refcor{PCtest} that it is possible to check in quadratic
time whether or not a  dag $(\GG, \DD)$ is 
a \PC. (One checks $\e(\LL_P)\geq 0$ for all nodes $P$,
\reflem{lem:integervalues} in the appendix.) 

\begin{example}\label{binarybasis}
We can represent every integer in the range 
$[-n,n]$ as the evaluation of some marking in a \PC with node set $\oneset{P_0 \lds P_\ell}$ such that $\e(P_i) =2^{i}$ for $0 \leq i \leq \ell$ and  $\ell = \floor{\log_2 n}$.
Thus, we can convert the binary notation of an integer $n$ into a \PC
with $\Oh(\log \abs n)$ vertices and $\Oh(((\log \abs n)\log\log \abs n)$ arcs.
\end{example}

\begin{example}\label{powtow}
A \PC can realize tower functions, since a line of $n+1$ nodes allows to represent
$\tow(n)$ as the evaluation of the last node. 
\end{example}

Sometimes it is convenient to think of a marking $M$ as a formal sum
$M = \sum_{P} M(P)P$. In particular, $-M$ denotes a marking with
$\e(-M) = -\e(M)$. 
For a marking $M$ we denote by $\sig(M)$ its \ei{support}, i.e., 
$$\sig(M)= \set{P\in \GG}{M(P)\neq 0}\sse \GG.$$ 

We say that $M$ is \ei{compact}, if we have 
$\e(P) \neq \e(Q)\neq 2 \e(P)$ for all $P,Q \in \sig(M)$, $P\neq Q$. 
If $M$ is compact, then we have $\e(M)=0$ \IFF $\sig(M)=\es$, and we have $\e(M)>0$ \IFF $M(P)$ is positive for the node
$P$ having the maximal value in $\sig(M)$. 
 
The insertion of a new node $\clone(P)$ without incoming arcs and with $\LL_{\clone(P)}=\LL_P$
is called \ei{cloning of a node} $P$. 
It is extended to markings, where $\clone(M)$ is obtained
by cloning all nodes in $\sig(M)$ and defining $M(\clone(P))=M(P)$ for $P\in\sig(M)$
and $M(\clone(P))=0$ otherwise.
We say that $M$ is a \ei{source}, if no node in $\sig(M)$ has any incoming arcs.
Note that $\clone(M)$ is always a source. 

If $M = \sum_{P} M(P)P$ and $K = \sum_{P} K(P)P$ are markings, 
then $M +K = \sum_{P} (M(P)+K(P))P$ is a formal sum where
coefficients $-2$ and $2$ may appear. For $M(P)+K(P)= \pm 2$,
let $P'=\clone(P)$.
We define a marking $(M+K)'$ by putting  $(M+K)'(P)= (M+K)'(P')= \pm 1$.
In this way we can realize addition (and subtraction) 
in a \PC by cloning at most $\abs{\sig(M) \cap \sig(K)}$ nodes.

Next, consider markings $U$ and $X$ with $\e(U)=u$ and $\e(X)= x$
such that $u2^x \in \Z$ (e.g. due to $x \geq 0$). We obtain a marking $V$ with $\e(V)=u2^x$
and $\abs{\sig(V)}= \abs{\sig(U)}$ as follows. First, let
$V=\clone(U)$ and $X'=\clone(X)$. Next, introduce additional arcs 
between all  $P'\in\sig(V)$ and $Q'\in X'$ 
with $\delta(P',Q')=X'(Q')$. 
Note that the cloning of $X$ avoids double arcs from $V$ to $X$. The cloning
of $U$ is not necessary, if $U$ happens to be a source. 

We now introduce an alternative representation for {\PC}s which allows us to
compare markings efficiently. 
The process of tranforming a $\Pi$ into this so-called tree representation is referred to
as \ei{reduction of a \PC}. 

\begin{definition}\label{def:treerep}
A \ei{tree representation} of a \PC $\Pi=(\GG,\del)$ consists of 
\begin{enumerate}[i)]
\item $\GG$ as a list $[P_1\lds P_n]$ such that $\e(P_i)<\e(P_{i+1})$ for all $1\le i<n$,
\item a bit vector $b(1)\lds b(n-1)$ where $b(i)=1$ \IFF $2\e(P_i)=\e(P_{i+1})$, and
\item\label{def:treerep:tree} a ternary tree of height $n$, where each node
has at most three outgoing edges, labeled by $+1$, $0$, and $-1$. 
All leaves are at level $n$, and each leaf represents the marking $M:\GG\to\OS$ given
by the labels of the unique path of length $n$ from the leaf to the root of the tree. These markings must be compact. 
Furthermore, all markings $\LL_P$ for $P\in\GG$ are represented by leaves. 
Finally, for each level in the tree $T$, we keep a list of the nodes in that level. 
\end{enumerate}
\end{definition}

If a path from some leaf to the root is labeled $(0,+1,-1,0,+1)$, then that leaf represents
the marking $P_2-P_3+P_5$ and we know (due to compactness) $\e(P_2)<2\e(P_3)<4\e(P_5)$.
The amount of memory for storing a tree representation is bounded by $\Oh(\abs{\GG}\cdot(\text{number of leaves}))$. 
Part \ref{def:treerep:tree}) of the definition is only needed inside the procedure \ET,
which is explained in the appendix. For simplicity the reader might want to ignore it in
a first reading and just think of a tree representation as the \PC graph plus an ordering
of the nodes and a bit vector keeping track of doubles.

\begin{proposition}[\cite{muw11pc}]\label{testintree}
There is a $\Oh(\abs{\GG})$ time algorithm which on input
a tree representation $\Delta$ of a \PC and two markings $K$ and $M$ (given as leaves)
compares $\e(K)$ and $\e(M)$. It outputs whether the two values are equal and if not,
which one of them is larger. In the latter case it also tells whether their difference
is $1$ or $\ge 2$. 
\end{proposition}

\begin{proof}
Start at the root of the tree and go down the paths corresponding to $K$ and $M$ in parallel. 
The first pair of different labels on these paths determines the larger value. 
The check whether $\e(K)=\e(M)+1$ is equally easy due to compactness. 
\qed\end{proof}

\begin{definition}\label{def:chain}
Let $\Pi=(\GG,\del)$ be a \PC. 
A chain (of length $r$) in $\Pi$ is a sequence of nodes $(P_0,P_1\lds P_r)$
where $\e(P_i)=2^i\e(P_0)$ ($0\le i\le r$). A chain is maximal if it is not part of a
longer chain. The number of maximal chains in $\Pi$ is denoted $\ch(\Pi)$.
We define the \ei{potential} of $\Pi$ to be $\pot(\Pi)=\ch(\Pi)\cdot\abs{\GG}$. 
\end{definition}

The following statement uses amortized time w.r.t.{} the potential function
$\pot(\Pi)$. Note that the potential $\pot(\Pi)$ remains bounded by $\abs{\GG}^2$ since $\ch(\Pi)\le\abs{\GG}$.

\begin{theorem}\label{extendtree}
The following procedure $\ET$ runs in amortized time
$\Oh((\abs{\GG}+\abs{U})\cdot\abs{U})$: 

Input: A dag $\Pi=(\GG\cupi U,\del)$, where $\GG$ and $U$ are disjoint 
with no arcs pointing from $\GG$ to $U$ and such that 
$(\GG,\del\vert_{\GG\times\GG})$ is a \PC in tree representation.
The potential is defined by the potential of its \PC-part $\pot((\GG,\del\vert_{\GG\times\GG})).$
The output of the procedure is "no", if $\Pi$ is not a \PC
(because $\e(P)\not\in\Z$ for some node $P$). 
In the other case, the output is a tree representation 
of a \PC $\Pi'=(\GG',\del')$ where:
\begin{enumerate}[i)]
\item\label{et:sub} $\GG \sse \GG'$ and $\del\vert_{\GG\times\GG}= \del'\vert_{\GG\times\GG}$.
\item\label{et:size} $\abs{\GG'}\le\abs{\GG}+3\abs{U}+(\ch(\Pi)-\ch(\Pi'))$
\item\label{et:map} For all $Q\in U$ there exists a node $Q'\in\GG'$ with $\e(Q)=\e(Q')$. 
\item\label{et:mark} For every marking $M$ in $\Pi$ there
exists a marking $M'$ in $\Pi'$ with $\e(M')=\e(M)$ and $\abs{\sig(M')}\leq\abs{\sig(M)}$.
\end{enumerate}
For $\sig(M)\sse\GG$ we can choose $M=M'$ by \ref{et:sub}).
If some further markings $M_1,M_2\lds M_m$ are part of the input
(those where $\sig(M)\cap U\neq \es$), we need additional amortized time
$\Oh(\abs{\sig(M_1)}+\cdots+\abs{\sig(M_m)}+m\cdot(\abs{\GG}+\abs{U}))$ to find the corresponding $M_1',M_2'\lds M_m'$. 
\end{theorem}

\begin{corollary}\label{maketree}
There is a $\Oh(\abs{\GG}^2)$ time procedure \MT that given a (graph representation) of
a \PC $\Pi=(\GG,\del)$ computes a tree representation. The number of nodes at most triples. 
\end{corollary}

\begin{proof}
This is a special case of \refthm{extendtree} when $U=\Pi$.
\qed\end{proof}

\begin{corollary}\label{PCtest}
The test whether a dag $(\GG,\del)$ defines a \PC 
can be done in $\Oh(\abs{\GG}^2)$. 
\end{corollary}

The efficiency of \MT is crucial for all our results. 
In particular, \refcor{maketree} improves the cubic time
complexity of \cite{muw11pc} for reduction to quadratic. 

\section{Arithmetic in the semi-direct product $\Z[1/2] \rtimes \Z$}\label{semi}

The basic data structure for this paper deals with the semi-direct product 
$\Z[1/2] \rtimes \Z$. Here $\Z[1/2]$ denotes the ring of rational numbers
with denominators in $2^\N$. Thus, an element in $\Z[1/2]$ is a rational number $r$ which can be written as $r = u2^x$
with $u,x \in \Z$. We view $\Z[1/2]$ as an abelian group with addition. Multiplication 
by $2$ defines an automorphism of $\Z[1/2]$, and hence the semi-direct product 
$\Z[1/2] \rtimes \Z$ becomes a (non-commutative) group where elements are 
pairs $(r,m) \in \Z[1/2] \times \Z$
and with the following explicit formula for multiplication: 
$$ (r,m) \cdot (s,n) = (r + 2^m s, m+n)$$
The semi-direct product 
$\Z[1/2] \rtimes \Z$ is also isomorphic to a group with two generators $a$ and $t$ and the 
defining relation $ta t^{-1}= a^2$. This group is known as the 
Baumslag-Solitar group $\BS(1,2)$. The isomorphism from $\BS(1,2)$ to 
$\Z[1/2] \rtimes \Z$ maps $a$ to $(1,0)$ and $t$ to $(0,1)$. This is a \homo due to 
$(0,1)(1,0)(0,-1)= (2,0)$. It is straightforward to see that it is actually 
bijective. 

We have $(r,m)^{-1} = (-r2^{-m}, -m)$ in $\Z[1/2] \rtimes \Z$, and a sequence of $s$ group operations
may lead to exponentially large or exponentially small values in the first component. 
Binary representation can cope with these values so there is no real 
need for \PC{}s when dealing with the group operation, only.

We equip $\Z[1/2] \rtimes \Z$ with a partially defined \ei{swap operation}.
For $(r,m) \in \Z \times  \Z \sse \Z[1/2] \rtimes \Z$ we define 
$\s(r,m) = (m,r)$. This looks innocent, but note that a sequence of $2^{\Oh(n)}$ defined
operations starting with $(1,0)$ may yield a pair $(0,\tow(n))$ where
$\tow$ is the tower function. Indeed 
$\s(1,0) = (0,1) = (0,\tow(0))$ and 
\begin{equation}\label{swaptower}
\s((0,\tow(n))(1,0)(0,-\tow(n)) = \s(\tow(n+1),0) = (0,\tow(n+1)).
\end{equation}

However, we will show in section \ref{app}:

\begin{theorem}\label{wpsdpz}
The \WP of the algebra $\sdpz$ with swapping is decidable in $\Oh(n^4)$. 
\end{theorem}

We use triples to denote elements in $\Z[1/2] \rtimes \Z$.
A triple $[u,x,k]$ with $u,x,k\in \Z$ and $x \leq 0 \leq k$ denotes the pair $(u2^x,k+x)\in \sdpz$. 
For each element in $\sdpz$  there are infinitely many corresponding triples. 
Using the generators $a$ and $t$ of $\BS(1,2)$ one can write: 
\begin{align*}
[u,x,k] = (u2^x,k+x) &= (0,x)(u,k) \in \Z[1/2]\rtimes\Z\\
& = t^x a^u t^k \in \BS(1,2)\text{ and}\\
[u,x,k] \cdot [v,y,\ell] &= [u 2^{-y} + v2^{k}, x+y, k +  \ell]
\end{align*}

In the following we use \PC{}s with \ei{triple markings} for elements 
in $\Z[1/2] \rtimes \Z$. We consider $T = [U,X,K]$, where 
$U,X,K$ are markings 
in a \PC 
with $\e(U) = u$ and  $\e(X) =x \leq 0 \leq \e(K)= k$; and 
we define
$\e(T)\in Z[1/2] \rtimes \Z$ to be the triple $\e(T) = [u,x,k] = (u2^x,x+k)$.

\section{Solving the \WP in the \baum group}\label{wpg}

The \baum group $\BG$ 
is a one-relator group with two generators $a$ and $b$ and the defining
relation $a^{a^b}=a^2$. (The notation $g^h$ means conjugation, here $g^h =h g h^{-1}$.
Hence $a^{a^b}= b a b^{-1} a b a^{-1} b^{-1}$.) 
The group $\BG$ can be written as an HNN extension of
$\BS(1,2)\simeq\Z[1/2]\rtimes\Z$ with \ei{stable letter} $b$; 
and $\BS(1,2)$ is an HNN extension of
$\Z\simeq \gen{a}$ with \ei{stable letter} $t$:
\begin{align*}
\langle a,b\mid a^{a^b}=a^2\rangle
&\simeq\langle a,t,b\mid a^t=a^2,a^b=t\rangle\cr
&\simeq\mathrm{HNN}(\langle a,t\mid a^t=a^2\rangle,b,\langle a\rangle\simeq\langle t\rangle)\\
&\simeq\mathrm{HNN}\left(\mathrm{HNN}(\langle a\rangle,t,\langle a\rangle\simeq\langle a^2\rangle),b,\langle a\rangle\simeq\langle t\rangle\right)
\end{align*}

Before the work of  Myasnikov, Ushakov and Won (\cite{muw11bg}) $\BG$ had been a possible candidate
for a one-relator group with an extremely hard (non-elementary) 
word problem in the worst case by the result of Gersten \cite{gersten91}. 
(Indeed, the tower function is visible as follows: Let $T(0) = t$ and $T(n+1) = b T(n) a T(n)^{-1}b ^{-1}$.
Then $T(n) = t^{\tow(n)}$ by a translation of \refeq{swaptower}.)
The purpose of this section is to improve the $\Oh(n^7)$ time-estimation of \cite{muw11bg} to cubic time. 
\refthm{wpbg} yields also the first practical algorithm 
to solve the \WP in the \baum group for a worst-case 
scenario\footnote{It is easy to design simple algorithms which perform
extremely well on random inputs. But for all these algorithms fail
on short instances, e.g. in showing $tT(6) = T(6)t$.}. 

\begin{theorem}\label{wpbg}
The \WP of the \baum group $\BG$ is decidable in time $\Oh(n^3)$. 
\end{theorem}

\begin{proof}
We assume that the input is already in compressed form given 
by a sequence of letters $b^{\pm 1}$ and pairwise disjoint \PC{}s each of them 
with a triple marking $[U,X,K]$ representing an element in $\sdpz$,
which in turn encodes a word over $a^{\pm 1}$'s and $t^{\pm 1}$'s. 

We use the following invariants:
\begin{enumerate}[i)]
\item $U,X,K$ have pairwise disjoint supports. 
\item $U$ is a source.
\item All incoming arcs to $X\cup K$ have their origin in $U$. 
\item Arcs from $U$ to $X$
have the opposite sign of the corresponding node-sign in $X$. 
\end{enumerate}

These are clearly satisfied in case we start with a sequence of 
$a^{\pm 1}$'s, $t^{\pm 1}$'s, and $b^{\pm 1}$'s. 
The formula $[u,x,k] \cdot [v,y,\ell] = [u 2^{-y} + v2^{k}, x+y, k + \ell]$ allows
to multiply elements in $\sdpz$ without destroying the invariants or increasing
the total number of nodes in the \PC{}s (the invariants make sure that cloning is not necessary). 
The total number of multiplications is bounded by $n$. Taking into account that there are at
most $n^2$ arcs, we are within the time bound $\Oh(n^3)$.

Now we perform from left-to-right \Breduction{}s, see \cite{LS01}. In terms of group
generators this means to replace factors $b a^{s} b^{-1}$ by $t^s$ and
$b^{-1} t^{s} b$ by $a^s$. Thus, if we see a subsequence 
$b [u,x,k] b^{-1}$, then we must check if
$x+k=0$ and after that if $u2^x\in\Z$. If we see a subsequence 
$b^{-1} [u,x,k] b$, then we must check
$u=0$. In the positive case we swap, in the other case we do nothing. 
Let us give the details: 
For a test we compute a tree representation of the circuit using \MT which takes time $\Oh(n^2)$. 
After each test for a \Breduction, the tree representation is deleted. 
There are two possibilities for necessary tests.
\begin{enumerate}[1.)]
\item $u = 0$. If yes,
remove in the original \PC the source $U$, this makes $X \cup K$ a source; replace
$[u,x,k]$ by $[x+k, 0, 0]$. The invariants are satisfied. 
\item $x+k = 0$. If yes, check whether $u2^x \in \Z$. If yes, replace $[u,x,k]$
in the original \PC by either 
$[0,u2^x,0]$ or $[0,0,u2^x]$ depending on whether $u2^x$ is negative or positive. 
We get $u2^x$ without increasing the number of nodes, since arcs from $U$ to $X$ have the opposite 
signs of the node-signs in $X$. Thus, if $E$ has been the set of arcs before 
the test, it is switched to $U \times X \setminus E$ after the test. 
The new marking for $u2^x$ is a source and does not introduce any cycle, because 
its support is still the support of the source $U$. 
\end{enumerate}
It is easy to see that computing a \Breduction on an input sequence of size $n$, 
we need at most $2n$ tests and at most $n$ of them
are successful. Hence we are still within the time bound $\Oh(n^3)$.

At the end we have computed in time $\Oh(n^3)$ a \Breduced normal form where
inner parts (i.e. the ones not involving $b^{\pm 1}$'s) are given as
disjoint \PC{}s. The result follows straightforwardly. \qed
\end{proof}

\section{Higman groups}

The Higman group $H_q$ has a finite presentation with 
generators $a_1 \lds a_q$ and defining relations 
$a_p a_{p-1} a_p^{-1} = a_{p-1}^2$ for all $p \in \Z/ q \Z$. From now on
we interpret indices $p$ for generators $a_p$ as elements of $\Z/q\Z$. 
In particular, $a_q = a_0$ and one of the defining relations says 
$a_1 a_{q} a_1^{-1} = a_{q}^2$. It is known \cite{serre80} that $H_q$ is trivial for $q \leq 3$ and infinite for $q \geq 4$.
Hence, in the following we assume $q \geq 4$. The group $H_4$ was the first 
example of a finitely generated group where all finite quotient groups 
are trivial. It has been another potential 
natural candidate for a group with an extremely hard (non-elementary) 
word problem in the worst case. Indeed, define:
\begin{align*}
w(p,0) &= a_p &\text{ for } p \in \Z / q \Z\\
w(p-1,i+1) &= w(p,i)a_{p-1} w(p,i)^{-1}&\text{ for } i \in \N\text{ and  }  \in \Z / q \Z
\end{align*}
By induction,  $w(p,n) = a_p^{\tow(n)}\in H_q$, where 
$\tow(n)$ is the $n$-th value of the tower function, but the length of the 
words  $w(p,n)$ is $2^{n+1} -1$, only. Hence there is a "tower-sized gap" between
input length and length of a canonical normal form.\footnote{This can be made more precise (and rigorous) by saying that the \ei{Dehn-function} of $H_4$ grows like a tower function (\cite{bridson10})}

For $i,j\in \N$ with $i \leq j$ we define
the group $G_{i \lds j}$ by the generators $a_i, a_{i+1},$ $ \ldots, a_j
\in \oneset{a_1 \lds a_q}$, and defining relations 
$a_p a_{p-1} a_p^{-1} = a_{p-1}^2$ for all $i<p\le j$.
Note that each $G_{i} \simeq \Z$ is the infinite cyclic group. 
The group $G_{1 \cdots q}$ is not $H_q$ because the relation 
$a_1 a_{q} a_1^{-1} = a_{q}^2$ is missing, but $H_q$ is a (proper) quotient
of  $G_{1 \cdots q}$. The groups $G_{i, i+1}$ are, by the very definition, isomorphic to 
the Baumslag-Solitar group $\BS(1,2)$, hence $G_{i, i+1}\simeq \Z[1/2] \rtimes \Z$.
It is also clear that $G_{i \lds j+1}\simeq G_{i \lds j}\ast_{G_j}G_{j,j+1} $ for $j-i < q$. Thus, $G_{123} \simeq G_{12}\ast_{G_2}G_{2,3}$ and
$G_{341} \simeq G_{34}\ast_{G_4}G_{41}$.

For simplicity we deal with $q=4$ only. The free group 
$F_{13}$, generated by $a_1$ and $a_3$ is a subgroup of $G_{123}$ as well as 
a subgroup of $G_{341}$, see e.g. \cite{serre80}. Thus we can build the amalgamated product
$G_{123}\ast_{F_{13}}G_{341}$ and a straightforward calculation shows 
$$H_4 \simeq G_{123}\ast_{F_{13}}G_{341}.$$
This isomorphism yields a direct proof that $H_4$ is an infinite group, see \cite{serre80}. 
In the following we use the following well-known facts about amalgamated products, see \cite{LS01,serre80,ddm10}. 
The idea is to calculate an alternating sequence of group elements 
{}from  $G_{123}$ and $G_{341}$. The sequence can be shortened, only
if one factor appears to be in the subgroup $F_{13}$. In this case we swap
the factor from $G_{123}$ to $G_{341}$ and vice versa. By abuse of language
we call this procedure again a \ei{\Breduction.} (This is perhaps no standard notation
in combinatorial group theory, but it conveniently unifies the same phenomenon in
amalgamated products and HNN-extensions; and the notion of \Breduction generalizes nicely to 
fundamental groups of graphs of groups.) 
Elements in the groups $G_{i,i+1}$ are represented by triple 
markings $T=[U,X,K]$ in some \PC{}. In order to remember that we evaluate $T$ in the group
$G_{i, i+1}$, we give each $T$ a \ei{type} $(i,i+1)$, which is denoted as a subscript.  
For  $\e(T) = [u,x,k]$ we obtain: 
\begin{align*}
\e(T_{(i, i+1)}) &= a_{i+1}^x a_i^{u}a_{i+1}^{k} &\in G_{i, i+1}\\
                 &=a_i^{u2^x}a_{i+1}^{x+k} &\text{ if $u2^x \in \Z$}
\end{align*}

The following basic operations are defined for all 
indices $i$, $i+1$, $i+2$ and $i\in \Z/4 \Z$, but for better readability we just
use indices 1, 2, and 3. 
\begin{itemize}
\item Multiplication: 
\begin{equation}\label{eqmul}
[u,x,k]_{(1,2)} \cdot [v,y,\ell]_{(1,2)} = [u 2^{-y} + v2^{k}, x+y, k +  \ell]_{(1,2)}.
\end{equation}
\item Swapping from $(1,2)$ to $(2,3)$: 
\begin{equation}\label{eqswap}
[0,x,k]_{(1,2)} = [x+k,0,0]_{(2,3)}.
\end{equation}
\item Swapping from $(2,3)$ to $(1,2)$: 
\begin{equation}\label{eqpswap}
[z,0,0]_{(2,3)} = [0,0,z]_{(1,2)} \text{ for } z \geq  0.
\end{equation}
\begin{equation}\label{eqnswap}
[z,0,0]_{(2,3)} = [0,z,0]_{(1,2)} \text{ for } z <  0.
\end{equation}
\item Splitting:
\begin{equation}\label{eq12split}
[u,x,k]_{(1,2)} = [u 2^x,0,0]_{(1,2)}\cdot [0,x,k]_{(1,2)} \text{ for } u 2^x \in \Z.
\end{equation}
\begin{equation}\label{eq23split}
[u,x,k]_{(2,3)} = [0,x,k]_{(2,3)}\cdot [u 2^{-k},0,0]_{(2,3)} \text{ for } u 2^{-k} \in \Z.
\end{equation}
\end{itemize}

From now on we work with a single \PC $\Pi$ together with a sequence 
$T_j$ ($j \in J$) of triple markings of various types. This is given 
as a tuple $\cT = (\GG, \del; (T_j)_{j \in J})$. 
We allow splitting operations only in combination with a multiplication, thus
we never increase the number of \tm{}s inside $\cT$. 
A tuple $\cT = (\GG, \del; (T_j)_{j \in J})$, where $(\GG,\del)$ is in
tree representation is called a \ei{\MDS}.
We keep $\cT$ as a \MDS by doing addition and multiplication by powers of $2$ using
clones and calling \ET on these after each basic operation. 

\begin{definition}\label{def:egin}
The \ei{weight} $\omega(T)$ of a \tm $T = [U,X,K]$ is defined as
$$\omega(T) = \abs{\sig(U)} + \abs{\sig(X)} + \abs{\sig(K)}.$$

The \ei{weight} 
$\omega(\cT)$ of a \MDS $\cT$ is defined as 
$\omega(\cT)= \sum_{j \in J} \omega({T_j}).$ 

Its \ei{size}\footnote{
The definition is justified, since 
we ensure $\abs{J} + \omega(\cT)\leq \Abs{\cT}$ whenever arguing about 
$\Abs{\cT}$.} $\Abs{\cT}$ is defined by $\Abs{\cT} = \abs \Gamma$.
\end{definition}

\begin{proposition}\label{gunnar}
Let $\cT = (\GG, \del; (T_j)_{j \in J})$ be a \MDS of size at most $m$, 
weight at most $w$ (and with $\abs{J} +w\leq m$). The following assertions hold.
\begin{enumerate}[i)]
\item No basic operation increases the weight of $\cT$. 
\item Each basic operation increases the size $\Abs{\cT}$ by $\Oh(w+(\ch(\Pi)-\ch(\Pi)))$. 
\item Each basic operation takes amortized time $\Oh(mw)$. 
\item\label{gunnar:total} A sequence of $s$ basic operations takes time $\Oh(smw+m^2)$
and the size of $\cT$ remains bounded by $\Oh(m+sw)$. 
\end{enumerate}
\end{proposition}

\begin{proof}
Applying a basic operation means replacing the left-hand side of the equation by the right-hand side,
thus forgetting any markings of the replaced triple(s). 
We can do the necessary tests, because we have a tree representation. For an operation we clone the involved markings,
but this does not increase the weight. Note that there is time enough to create the clones with all their outgoing arcs. 
This yields the increase in the size by $\Oh(w)$. 
With the new clones we can perform the operations by using the algorithms
described in section \ref{PCs} on the graph representation of the circuit. 
We regain the \MDS by calling \ET which integrates the modified
clones into the tree representation. 

In order to get \ref{gunnar:total}) we observe that the initial number of maximal chains
is $m$ and there are at most $\Oh(w)$ new ones created in each basic operation. Hence the total
increase in size is $\Oh(sw)$ and the difference in potential is at most $m(m+sw)$. The
time bound follows. 
\qed\end{proof}

\subsection{Solving the \WP in Higman's group}\label{wph4}

\begin{theorem}\label{wph}
The \WP of $H_4$ can be solved in time $\Oh(n^6)$. 
\end{theorem}

The rest of this section is devoted to the proof of \refthm{wph}. 
For solving the word problem in the Higman group $H_4$ the traditional input is 
a word over generators  $a_{p}^{\pm 1}$.
We solve a slightly more general problem by assuming that the  input consists of a single
\PC $\Pi = (\Gamma, \del)$ together with a sequence of  $s$ triple markings of various types. 
Each \tm $[U,X,K]_{(p,p+1)}$ corresponds to  $a_{p+1}^{\e(X)}a_{p}^{\e(U)}a_{p+1}^{\e(K)} \in H_4$. 

Let us fix $w$ to be the total weight of 
$\cT=(\GG, \del; (T_j)_{1 \leq j \leq s})$. 
For simplicity we assume $s \leq w$ and that $w$ and sizes of clones are bounded by $\Abs{T}=\abs{\GG}$. (This
is actually not necessary, but it simplifies some bookkeeping.) Having $s \leq w \leq n \in \Oh(w)$, we can think of $n = \abs\GG$ as our input size. 
We transform the input $\cT = (\GG, \del; (T_j)_{1 \leq j \leq s})$ into a \MDS by a call of \MT. 

During the procedure $\abs\GG$ increases, but the number
of \tm{}s remains bounded by $s$ and the weight remains bounded by $w$.

In order to achieve our main result we show how to solve the word 
problem with $\Oh(s^2)$ basic operations on the \MDS $\cT$. 
Assume we have shown this. Then, by \refprop{gunnar}, the final size will be bounded by 
$m \in \Oh(s^2w)$; and the time for all basic operations is therefore $\Oh(s^4w^2)\sse\Oh(n^6)$.

We collect sequences of \tm{}s of type $(1,2)$ and $(2,3)$ 
in \ei{intervals} $\cL$, 
which in turn receive type $(1,2,3)$; and we 
collect \tm{}s of type $(3,4)$ and $(4,1)$ in intervals of type $(3,4,1)$.
Each interval has (as a sequence of \tm{}s) a semantics $\eps(\cL)$ which is
a group element either in $G_{123}$ or in $G_{341}$ depending on the type
of $\cL$. Thus, it makes sense to ask whether $\eps(\cL) \in F_{13}$.
These tests are crucial and dominate the runtime of the algorithm. 

Now the sequence $(T_j)_{1 \leq j \leq s}$ of triple markings appears as 
a sequence of intervals: 
$$(\cL_1 \lds \cL_{f};  \cL_{f+1} \lds \cL_t).$$
We introduce a separator ";" dividing the list in two parts. 

The following invariants are kept up:
\begin{enumerate}[i)]
\item All $\cL_1 \lds \cL_{f}$ satisfy $\eps(\cL_i)\notin F_{13}$. In particular, these
intervals are not empty and they represent non-trivial group elements
in $(G_{123}\cup G_{341}) \sm F_{13}$.
\item The types of intervals left of the separator are alternating. 
\end{enumerate}

In the beginning each interval consists of exactly one \tm, thus $f=0$ and $t=s$. 
The algorithm will stop either with $1 \leq f = t$ or with $f=0$ and $t=1$.

Now we describe how to move forward:
Assume first $f=0$. (Thus, $t>1$.) If $\eps(\cL_{1}) \notin F_{13}$, then move the separator to the right, i.e. we obtain 
$f=1$. 
If $\eps(\cL_{1}) \in F_{13}$, then, after possibly swapping $\cL_{1}$, we join the intervals $\cL_{1}$ and $\cL_{2}$
into one new interval. In this case we still have $f=0$, but $t$ decreases by 1.

From now on we may assume that $0<f < t$. 
If $\cL_{f}$ and  $\cL_{f+1}$ have the same 
type, then append $\cL_{f+1}$ to $\cL_{f}$, and move the separator to the left of $\cL_f$. 
Thus, the values $f$ and $t$ decrease by 1. 

If $\cL_{f}$ and  $\cL_{f+1}$ have different types, 
then we test whether or not $\eps(\cL_{f+1}) \in F_{13}$. 
If $\eps(\cL_{f+1}) \notin F_{13}$, then move the separator to the right, i.e.  
$t-f$ decreases by $1$. 
If $\eps(\cL_{f+1}) \in F_{13}$, then we swap 
$\cL_{f+1}$ and join the intervals $\cL_{f}$ and $\cL_{f+1}$
into one new interval. Since we do not know whether the new interval belongs
to $F_{13}$, we put the separator in front of it, decreasing both $f$ and $t$ by 1. 

We have to give an interpretation of the output of this algorithm. 
Consider the case that we terminate with $1 \leq f = t$. 
Then $\eps(\cL_1) \cdots \eps( \cL_{t}) \in H_4$ is a \Breduced 
sequence in the amalgamated product. It represents a non-trivial group
element, because $t\geq 1$. 

In the other case we terminate with $f=0$ and $t=1$. We will make sure that
the test "$\eps(\cL) \in F_{13}$?" can as a by-product also answer the
question whether or not $\eps(\cL)$ is the trivial group element.
If we do so, one more test on $(\cL_1)$ yields the answer we need. 

Now, we analyze the time complexity. 
Termination is clear
once we have explained how to implement a test "$\eps(\cL) \in F_{13}$?". 
Actually, it is obvious that the number of these tests is bounded by $2s$.
Thus, it enough to prove the following claim. 

\begin{lemma}\label{lem:test13}
Every test "$\eps(\cL) \in F_{13}$?" can be realized 
with  $\Oh(s)$ basic operations in the \MDS $\cT$. The test yields either "no" or it says "yes" with the
additional information whether or not $\eps(\cL)$ is the trivial group 
element. Moreover, in the "yes" case we can also swap the type of $\cL$ within the same  bound on basic operations. 
\end{lemma}

\begin{proof}
Let us assume that $\cL$ is of type $(1,2,3)$, i.e., it contains only triples of types $(1,2)$ and $(2,3)$. 
Let $s$ be the length of $\cL$. 
The group $G_{123}$ 
is an amalgamated product where $F_{13}$ is a free subgroup 
of rank 2, see \cite{serre80} for a proof. 
In a first round we create a sequence of \tm{}s  
$$(T_1 \lds  T_t)$$
with $t \leq s$ such that for $1 \leq i < t$ the type of $T_i$ is $(1,2)$ \IFF the type of $T_{i+1}$ 
is $(2,3)$. We can do so  by $s-t$ basic multiplications from left-to right without changing the
semantics of $g = \eps(T_1) \cdots \eps(T_t)\in G_{123}$. 
 
Next, we make this sequence \Breduced. Again, we scan from left to right. 
If we are at $T=T_i$ with value $[u,x,k]$ we have to check that 
either $[u,x,k]_{(1,2)}  = (0,z)\in \sdpz$ or $[u,x,k]_{(2,3)} = (z,0)\in \sdpz$
for some integer $z\in \Z$. 
 
For the type $(1,2)$ we have $[u,x,k]_{(1,2)}= (0,z)$ \IFF $u=0$, which in a tree representation
means that the support of the marking for $u$ is empty. Hence this test
is trivial. If the test is positive, we can replace  $[u,x,k]_{(1,2)}$ by $[0,x,k]_{(1,2)}$
and we perform a swap to type $(2,3)$. If $t>1$ we can recursively perform
multiplications with its neighbors, thereby decreasing the value $t$. 
 
For  the type $(2,3)$ we have $[u,x,k]_{(2,3)}= (z,0)$ \IFF 
both $k+x=0$ and $u2^x \in \Z$. Tests are possible in linear time and 
if successful, we continue as in the precedent case. 
 
The final steps are more subtle.  Let $\e(T_j) = g_j \in G_{12} \cup G_{23}$.
Recall that $(g_1 \lds g_t)$ is already a \Breduced sequence. 
We have $g_1 \cdots g_t \in F_{13}$ \IFF there is a sequence
$(h_0, h_1 \lds h_t)$ with the following properties:
\begin{enumerate}[i)]
\item $h_0 = h_t = 1$ and $h_j \in G_2$ for all $0\leq j \leq t$. 
\item $h_{j-1} g_j = g'_j h_{j}$ with $g'_j \in G_{1} \cup G_{3}$
for all $1\leq j \leq t $.
\end{enumerate}

Assume that such a sequence $(h_0, h_1 \lds h_t)$ exists. Then we have 
$g'_j \in G_1$ \IFF $g_j \in G_{12}$. Moreover, whenever
$gh= g'h'\in G_{123}$ with $g,g' \in G_1 \cup G_3$ 
and 
$h, h' \in G_2$, then $g = g'$ and $h=h'$. 
This follows because  $ g'^{-1}g  =h'h^{-1}  \in  F_{13}\cap G_2 = \oneset{1}$.
Thus, the product
$h_{j-1} g_j$ uniquely defines $g'_j \in G_{1} \cup G_{3}$ and $ h_{j}\in G_2$,
because $h_0 = 1$ is fixed. 

The invariant during a computation from left to right is that
$\eps(T_j) = h_{j-1} g_j$. We obtain $\eps(T_j) = g_{j}' h_j$ by a basic
splitting. If no splitting is possible we know that  $g \notin F_{13}$ and we can stop.
If however a splitting is possible, then we have two cases. 
If $j$ is the last index ($j=t$), then, in addition, we must have $h_j=1$. We can test this. 
If the test fails, we stop with $g\notin F_{13}$. 
If we are not at the last index we perform a swap. We split, then swap the 
right hand factor and multiply it with the next \tm, which has the correct type 
to do so. As our sequence has been \Breduced the total number of \tm{}s 
remains constant. There can be no cancelations at this point. 
Thus, the test gives us the answer to "$\eps(\cL)\in F_{13}$?" using $\Oh(s)$ basic operations.
In the case $\eps(\cL) \in F_{13}$ we still need to know 
whether $\eps(\cL) = 1 \in G_{123}$. For $t> 1$ the answer is "no". 
It remains to deal with $t=1$. But a test whether $\eps([u,x,k])=1 $ just means
to test both $u=0$ and $x+k=0$. 

Now, assume we obtain a "yes" answer and we know $\eps(\cL)\in F_{13}$.
We do the swapping of types from left to right by using only the 
left factor in a splitting.  These are additional $s$ basic operations, 
hence the total number of $\Oh(s)$ did not increase. \qed
\end{proof}

\section{Conclusion and future research}\label{conclusion}
The \WP is a fundamental problem in algorithmic group theory. 
In some sense "almost all" finitely presented groups are hyperbolic and 
satisfy a "small cancelation" property, so the \WP is solvable in linear time! 
For hyperbolic groups there are also efficient parallel algorithms and 
the \WP is in $\mathbf{NC^2}$, see \cite{cai92stoc}.
On the other hand, for many naturally defined groups little is known. 
Among one-relator groups the \baum group $\BG$ was supposed to have the hardest
\WP. But we have seen that it can be solved in cubic time. The method generalizes 
to the higher \baum groups $G_{(m,n)}$ in case that $m$ divides $n$, but this
requires more "\PC machinery" and  
has not worked out in full details yet, see \cite{muw11bg}. The situation for 
$G_{(2,3)}$ is open and related to questions in number theory. 
The Higman groups $H_q$ belong to another family of naturally appearing groups 
where the \WP was expected to be non-polynomial. We have seen that the \WP in $H_4$
is in $\Oh(n^6)$. It easy to see that our methods show that the \WP in $H_q$
is always in {\bf P}, but to date the exact time complexity has not been analyzed for $q>4$. 

\baum and Higman groups are built up via simple HNN extensions and amalgamated products.
Many algorithmic problems are open for such constructions, for advances 
about theories of HNN-extensions and amalgamated products we refer to \cite{LohSen06}. 

Another interesting open problem concerns the \WP in \ei{Hydra} groups. 
Doubled hydra groups have Ackermannian Dehn functions \cite{Riley2010arXiv1002.1945D},
but still it is possible that their \WP is solvable in polynomial time. 

\newpage
\addcontentsline{toc}{section}{Bibliography}
\newcommand{\Ju}{Ju}\newcommand{\Ph}{Ph}\newcommand{\Th}{Th}\newcommand{\Ch}{C%
h}\newcommand{\Yu}{Yu}\newcommand{\Zh}{Zh}

\newpage
\section*{Appendix}\label{app}

\subsection*{Reduction of \PC{}s}
 
In this section we give a full proof of \refthm{extendtree}. 
We start with the observation that due to compactness of the markings in a tree
representation, the leaves are automatically ordered by $\e$-value. 
The leftmost has the smallest and the rightmost has the largest $\e$-value.
This easy calculation (essentially an argument about binary sums) is left to the reader. 

Next, we establish some operations on tree representations. 

\begin{lemma}{(Insertion of a new node)}\label{lem:insertnode}
Let $\Pi=(\GG, \del)$ be a \PC in tree representation and $M$ a marking in $\Pi$
given as a leaf in the tree. Then we can in amortized time $\Oh(\abs{\GG})$ insert
a new node $P$ into the circuit with $\LL_P=M$. 
\end{lemma}

\begin{proof}
If a node $P$ with $\e(P)=2^{\e(M)}$ already exists, we abort immediately. Otherwise
the index of the leaf corresponding to the (compact) marking $M$ inside the list of leaves tells
us the position of $P$ in the sorting of $\GG$. (Actually, we have to count the
number of leaves of type $\LL_Q$ ($Q\in\GG$) that are left of $M$.)

Next, we need to update the bit vector $b$. This is achieved by the procedure
described in \refprop{testintree}, which tells us whether $\e(M)+1=\e(\LL_Q)$ where
$Q$ is the node succeeding $P$ in the ordering of $\GG$. 

Finally, we have to "stretch" the tree defined by $\DD$ by inserting a new level corresponding
to the new node $P$. All the edges on that level have to be labeled by "0", as no
marking uses the newly created node yet. 
This can be done in linear time using the lists of nodes we keep for each level. 

Note that this may increase the potential by $\abs{\GG}$, since both
$\abs{\GG}$ and the number of chains might grow by one. This adds
$\abs{\GG}$ to the amortized time, which is captured by the $\Oh$-notation. 
\qed\end{proof}

\begin{lemma}{(Incrementation of a marking)}\label{lem:incmarking}
Given a marking $M$ in a tree representation of $\Pi=(\GG,\del)$ one can generate
in $\Oh(\abs{\sig(M)})$ time a marking $M'$ with $\e(M')=\e(M)+1$. 
\end{lemma}

\begin{proof}
If $\GG$ is empty, the claim is obvious. Otherwise
let $P$ be the unique node in $\GG$ with $\e(P)=1$. If $M(P)\neq +1$, we
increment $M(P)$ by one and are done. Otherwise, we look for a node $P'$ with
$\e(P')=2$. If it doesn't exist (in which case $\GG=\{P\}$ and $\e(M)=1$), we
create it and put $M(P)=0$ and $M(P')=+1$. If $P'$ does exist, then $M(P')=0$
due to compactness. Again, put $M(P)=0$ and $M(P')=+1$. 

Note that in the last case the newly created marking is not necessarily compact
and therefore cannot be inserted immediately into the tree representation as a leaf. 
We will deal with this in the next lemma. 
\qed\end{proof}

\begin{lemma}{(Making a marking compact)}\label{lem:compactify}
Let $\Pi=(\GG,\del)$ be a \PC in tree representation and $M$ be a marking in $\Pi$
(not yet a leaf and in particular not compact; e.g. given as a list of signed
pointer to nodes). Assume that for each node $P\in\sig(M)$ the last node $T$ in the
longest chain starting in $P$ is not marked by $M$. 
Then $M$ can be made compact in time $\Oh(\abs{\sig(M)})$ (and after that integrated
into $\Pi$ as a leaf in time $\Oh(\abs{\GG})$) without changing the circuit and
without increasing the weight of $M$. 
\end{lemma}

\begin{proof}
We look at the nodes $P\in\sig(M)$ in ascending order (w.r.t. their value).
There are essentially two ways for $M$ not to be compact at a point $P$: 
Assume that $M(P)=+1$ ($-1$ is similar). 
\begin{enumerate}[1.)]
\item $P$ is the first node in a chain length $2$ which $M$ labels $(+1,-1)$. 
Replace it by $(-1,0)$. 
\item $P$ is the first node of a chain $(P=P_1,P_2\lds P_k,P_{k+1})$ labeled
$M(P_i)=+1$ ($1\le i\le k$) and $M(P_{k+1})\neq +1$. Note that by assumption $P_{k+1}$ exists. 
Replace the labels by $(-1,0\lds 0,+1)$. We might need to repeat this (if there
is a node $P_{k+2}$ with $\e(P_{k+2})=2\e(P_{k+1})$ and $M(P_{k+2})=+1$) but this
ultimately stops at $T$. 
\end{enumerate}
\qed\end{proof}

Often we need to increment a leaf marking by one and make it compact. In order to have
the necessary nodes, we introduce the following concept:

\begin{definition}\label{def:joker}
Let $\Pi= (\GG, \del)$ be a power circuit in tree representation. A node $J$
is called a \ei{joker}, if it is the last in a maximal chain starting at
the unique node with value $1$ and $J$ is not used in any leaf marking.
\end{definition}

\begin{lemma}\label{lem:createjoker}
Let $\Pi=(\GG,\del)$ be a power circuit in tree representation.
Then we can in amortized time $\Oh(\abs{\GG})$ insert
a joker into the circuit. 
\qed\end{lemma}

\begin{proof}
Start at the node $P_0$ with $\e(P_0)=1$. Using the bit vector, find the first
"gap" in the chain starting at $P_0$, i.e., the largest $i$ such that
$P_0\lds P_{n-1}$ exist with $\e(P_i)=2^i$. Keep the number $n$ in binary notation. 
Like in \reflem{lem:compactify} compute a compact representation of $n$. Note that
we are dealing with ordinary numbers here, not circuits! Use the compact
representation of $n$ to create $P_n$ with $\e(P_n)=2^n$. Check whether there is
a node $P_{n+1}$ and adjust the bit vector. If yes, $P_n$ linked two maximal chains,
so in amortized analysis we don't have to account for the $\Oh(\abs{\GG})$ time
used so far. Repeat the process until we create a node that is the end of a maximal
chain. This is the joker. 
\qed\end{proof}

Now we are ready to prove \refthm{extendtree}. 

\begin{proofof}{\refthm{extendtree}}{
Let $n=\abs{\GG}+\abs{U}$. We may assume $n\ge 1$. 

Perform a topological sorting of $U$, i.e. find an enumeration
$U =\oneset{Q(0)\lds Q(\abs{U}-1)}$ such that there are
no arcs from any $Q(i)$ to $Q(j)$ when $i\le j$. 
Since $\Pi$ is a dag, a topological ordering of $U$ exists
and it can be found in time $\Oh(n\abs{U})$, see e.g. \cite{CLRS09}. 
The nodes of $U$ will be moved to $\GG$ in ascending topological order,
so that all the time $\GG$ remains a circuit, i.e., there are
no arcs from $\GG$ to $U$. 

Let $M$ be any one of the markings $M_j$ ($j=1\lds m$) or $\LL_P$ ($P\in U$). 
While the nodes of $U$ are being moved to $\GG$, there may be times when
the support of $M$ is partly in $\GG$ and partly still in $U$. Later it will be
completely contained in $\GG$. We will maintain the following invariants:
\begin{enumerate}[i)]
\item Any marking whose support is completely contained in $\GG$ is represented as a leaf (and thus compact).
\item\label{et:topinv} For all other markings $M$ and all nodes $P\in\sig(M)\cap\GG$
there is a chain starting at $P$ and ending at a node $T\not\in\sig(M)$ such that
there is node node with double the value of $\e(T)$ (i.e. the chain cannot be
prolonged at the top end).
\end{enumerate}

We now describe how to do the moving of the topologically smallest node $Q$ of $U$. 

We have $\LL_Q\sse\GG$, so by the invariant, this is a leaf. 
Hence it is possible to test whether $\e(\LL_Q)<0$ using \refprop{testintree}. 
If this is the case, $\Pi$ is not a \PC and we stop with the output "no". 
From now on we assume $\e(Q)\ge 1$.

\begin{enumerate}[1.)]
\item\emph{Insert a new joker}\\
See \reflem{lem:createjoker}. Now we don't have to worry about incrementing
compact markings anymore. 
\item\emph{Find a replacement $P$ for $Q$ in $\GG$}\\
Check whether there is a node $P\in\GG$ with value $\e(P)=\e(Q)$. If not, use
\reflem{lem:insertnode} to create it taking the marking $\LL_Q$ as $\LL_{P}$. 
This takes amortized time $\Oh(\abs{\GG})$.
\item\emph{Adapt markings using $Q$:}\\
Using the bit vector, go up the chain in $\GG$ starting at $P$. Prolong the chain
at the top by creating a new node $P'$ (use \reflem{lem:incmarking} on the
successor marking of the last node of the chain, insert it into the tree via
\reflem{lem:compactify} and use it as a successor marking for creating the
new node $P'$ with \reflem{lem:insertnode}. The time needed is $\Oh(\abs{\GG})$. \\
Go through all markings $M$ ($\LL_{Q'}$ for $Q'\in U$ and $M_j$ for $j=1\lds m$)
that have $Q\in\sig(M)$. Replace $Q$ by $P$ in $M$. If this leads to a double
marking of $P$ by $M$, replace those by the next node in the chain. Again, that
node might become doubly marked by $M$, so repeat this. This stops at the latest
at $P'$ which is new and thus unmarked $M$. For each of these steps, the support
of $M$ decreased by one, so the total time (for all $Q\in U$) is bounded by
the sum of the sizes of all supports, i.e., $n\cdot\abs{U}$ for successor
markings and $\abs{\sig(M_1)}+\cdots+\abs{\sig(M_m)}$ for the markings $M_j$ ($j=1\lds m$). 
Now $Q$ is not part of any marking anymore and can be deleted. 
\item\emph{Make markings compact:}\\
If $Q$ was the last node of a marking $M$ to be moved from $U$ to $\GG$, we have
to make $M$ compact and create a leaf in the tree. This is done by using
\reflem{lem:compactify}. Note that we have the invariant \ref{et:topinv}). 
We need time $\Oh(\abs{\sig(M)}+n)$, which over the whole procedure sums up to 
$\Oh(n\cdot\abs{U})$ for successor markings and 
$\Oh(\abs{\sig(M_1)}+\cdots+\abs{\sig(M_m)}+m\cdot n)$ for the markings $M_j$ ($j=1\lds m$). 
\item\emph{Make room for later compactification:}\\
Start at $P'$ and create a new node $T$ that has value $\e(T)=2\e(P')$. Check
whether there is a node with value $2\e(T)$. Is yes, the creation of $T$ has
linked two maximal chains, thus decreasing the potential by $\abs{\GG}$. This
pays for the $\Oh(\abs{\GG})$ time needed for creating $T$ and the check. 
Repeat this until we create a $T$ that has no node with double the value of $T$.
Note that this takes only amortized time $\Oh(\abs{\GG})$. 
\end{enumerate}
}\end{proofof}

\subsection*{A more detailed look at \PC{}s}

\begin{lemma}\label{lem:integervalues}
The following assertions are equivalent: 
\begin{enumerate}[1.)]
\item $\e(P)\in 2^{\N}$ for all nodes $P$, 
\item $\e(\LL_P)\geq 0$ for all nodes $P$,
\item $\e(M)\in \Z$ for all markings $M$.
\end{enumerate}
\end{lemma}

\begin{proof}
Choose some node $P$ without incoming arcs which exists 
because $(\GG,\del)$ defines a dag. The assertions are equivalent
on $\GG\sm\oneset{P}$ by induction. The result now follows easily. 
\end{proof}

\begin{proofof}{\refcor{PCtest}}{
The procedure \MT uses \ET which detects if $\e(P)\notin 2^\N$.
The result follows from \reflem{lem:integervalues}. 
}\end{proofof}

\subsection*{The \WP of the algebra $\sdpz$ with swapping}

\begin{proofof}{\refthm{wpsdpz}}{
The proof is almost the same as the proof of \refthm{wpbg} which has been 
given above. Hence we focus on the differences in the proof. The input is given 
by a sequence of pairwise disjoint \PC{}s each of them 
with a triple marking $[U,X,K]$ representing an element in $\sdpz$. We use 
only the  invariant that 
$U,X,K$ have pairwise disjoint supports. 
Swapping of $[u,x,k]$ is possible, if $z= u2^x \in \Z$ and the 
result is either $[x+k,z,0]$ or $[x+k,0,z]$ depending on the sign of $z$. 
In order to realize a marking for $z$ we clone $U$. This increases the 
size by $\abs{\sig(U)}$. At the end the size of the \PC is quadratic in the input. 
This yields $\Oh(n^4)$ time. 
}\end{proofof}

\subsection*{Dehn functions}\label{sec:dehn}
No result about Dehn functions is used in our paper. However, for convenience
of the interested reader we recall the definition of a Dehn function
as given by Wikipedia.  
Let $G$ be given by a finite generating set $X$ with a finite defining 
set of relations $R$. 
Let $F(X)$ be the free group with basis $X$ and let $w \in  F(X)$ be a relation in $G$, that is, a freely-reduced word such that 
$w = 1$ in $G$. Note that this is equivalent to saying that is, $w$ belongs to the normal closure of $R$ in $F(X)$.
Hence we can write $w$ as a sequence of $m$ words $xr x^{-1}$ 
with $r \in R^{\pm 1}$ and $x \in F(X)$.
The \ei{area} of $w$, denoted  $\mathrm{Area}(w)$, is the smallest $m\geq 0$
such that there exists such a representation  for $w$ as the product in $F(X)$ of $m$ conjugates of elements of $R^{\pm 1}$ .
Then the Dehn function of a finite presentation $G = \gen{X\mid R}$ is defined as
$$\mathrm{Dehn}(n) = \max\set{\mathrm{Area}(w)}{ w= 1 \in G, \; \abs w \leq n, \text{ and $w$ is freely-reduced}}.$$
Two different finite presentations of the same group are equivalent with respect
to \ei{domination}. Consequently, for a finitely presented group the growth type of its Dehn function does not depend on the choice of a finite presentation for that group. 

A function $f : \N \to \N$  is dominated by $g: \N \to \N$, if there exists 
$c \geq 1$  such that $f(n) \leq c g(cn +c ) + cn + c$ for all $n \in \N$. 

\end{document}